\documentclass[12pt]{amsart}
\usepackage[margin=3cm]{geometry}

\usepackage[utf8]{inputenc}
\usepackage[english]{babel}
\usepackage{amsmath,amssymb,amsthm}
\usepackage[abbrev]{amsrefs}
\usepackage{enumerate,mathrsfs}

\nocite{*}

\newtheorem{theorem}{Theorem}[section]
\newtheorem{corollary}[theorem]{Corollary}
\newtheorem{lemma}[theorem]{Lemma}
\newtheorem{proposition}[theorem]{Proposition}

\theoremstyle{definition}

\newtheorem{remark}[theorem]{Remark}

\DeclareMathOperator{\Aut}{Aut}

\DeclareMathOperator{\id}{id}

\newcommand{\bbF}{\mathbb{F}}
\newcommand{\edge}[1]{\stackrel{#1}{\longrightarrow}}

\title[Graphical Frobenius representations]{Graphical Frobenius representations of non-abelian groups}
\author{G\'abor Korchm\'aros}
\address{Dipartimento di Matematica, Informatica ed Economia\\
	Universit\`a della Basilicata\\
	Contrada Macchia Romana\\
	85100 Potenza, Italy}
\email{gabor.korchmaros@unibas.it}
\author{G\'abor P. Nagy}
\address{Department of Algebra \\
	Budapest University of Technology and Economics\\
	Egry J\'ozsef utca 1\\
	H-1111 Budapest, Hungary}
\address{Bolyai Institute \\
	University of Szeged \\
	Aradi v\'ertan\'uk tere 1\\
	H-6720 Szeged, Hungary}
\email{nagyg@math.bme.hu}
\thanks{Support provided by NKFIH-OTKA Grants 114614, 115288 and 119687.}
\date{Version 06/09/2019.}



\begin{document}

\begin{abstract}
A group $G$ has a Frobenius graphical representation (GFR) if there is a simple graph $\varGamma$ whose full automorphism group is isomorphic to $G$ and it acts on vertices as a Frobenius group. In particular, any group $G$ with GFR is a Frobenius group and $\varGamma$ is a Cayley graph. The existence of an infinite family of groups with GFR whose Frobenius kernel is a non-abelian $2$-group has been an open question. In this paper, we give a positive answer by showing that the Higman group $A(f,q_0)$ has a GFR for an infinite sequence of $f$ and $q_0$. 
\end{abstract}

\maketitle

\section{Introduction} Graphs and their automorphism groups have intensively been investigated especially
for vertex-transitive (and hence regular) graphs. Many contributions have concerned vertex-transitive graphs with large automorphism groups compared to the degree of the graph, and have in several cases relied upon deep results from Group theory, such as  the classification of primitive permutation groups.

On the other end, the smallest vertex-transitive  automorphism groups of graphs occur when the group is regular on the vertex-set. A group is said to have a graphical regular representation (GRR) problem if there exists a graph whose (full) automorphism group is isomorphic to $G$ and acts regularly on the vertex-set. Actually, almost all finite groups have (GRR). In fact, all the few exceptions were found in the 1970-80s by a common effort of G. Sabidussi, W. Imrich, M.E. Watkins, L.A. Nowitz, D. Hetzel, C.D. Godsil, and L. Babai, see \cite{MR3864735}*{Section 1}. Since regular automorphism groups of a graph are those which are vertex transitive but contain no non-trivial automorphism fixing a vertex, a natural next choice as a small vertex-transitive automorphism group of a graph may be a Frobenius group on the vertex-set: an automorphism group of a graph that is vertex-transitive but not regular and only the identity fixes more than one vertex. It is well known that every group may be a Frobenius group in at most one way. Furthermore, each graph $\Gamma$ with a (sub)group $G$ of automorphisms acting regularly on the vertex-set is a Cayley graph $\mathrm{Cay}(\Gamma(G,S))$.

All these give a motivation for the study of Frobenius groups $G$ which have a graphical Frobenius representation (GFR), that is, there exists a graph whose (full) automorphism group is isomorphic to $G$ and acts on the vertex-set as a Frobenius group. The systematic study of the GFR problem was initiated by J.K. Doyle, T.W. Tucker and M.E. Watskin  in their recent paper \cite{MR3864735}. As it was pointed out by those authors, the GFR problem is largely not analogous to the GRR problem since all groups have a regular representation whereas  Frobenius groups have highly restricted algebraic structures, and many large classes of abstract groups are not Frobenius groups. It is apparent from the results, examples and classification of smaller groups with GFR in \cite{MR3864735}, see in particular \cite{MR3864735}*{Theorem 5.3 and Remark 5.4}, that an interesting open question is the existence of a (possible infinite) family of Frobenius groups with GFR whose kernel is a non-abelian $2$-group.

In this paper we give an affirmative answer to that question. Our choice of Frobenius groups is influenced by Higman's classification of Suzuki $2$-groups \cite{hig}, as we take for $G$ the group $A(f,q_0)$ from Higman's list where $q_0$ and $q=2^f$ are $2$-powers. $A(f,q_0)$ is a subgroup of $G$ of $GL(3,\mathbb{F}_q)$ whose main properties are recalled in Section \ref{sec1}. 
We build a Cayley graph $\Gamma_u$ on the Frobenius kernel $K$ of $G$, with a certain inverse closed subset $S$ of $K$ as generating set, constructed from an element $u\in \mathbb{F}_q$.  We show that $G$ has GFR on $\Gamma_u$ provided that $q=2^f$, $q_0$ and $u$ are carefully chosen.

Our notation and terminology are standard. For the definitions and known results on Frobenius groups which play a role in the present paper, the reader is referred to  \cite{MR3864735}.

\section{The group $A(f,q_0)$}
\label{sec1}
Let $\mathbb{F}_q$ be the finite field of order $q=2^f$ with $f\ge 4$, and let $q_0=2^{f_0}$ be another power of $2$ smaller than $q$. For $a,c\in \bbF_q$ we write
\[\varPhi_{a,c}=\begin{bmatrix} 1&0&0 \\ a&1&0 \\ c&a^{q_0}&1 \end{bmatrix}, \qquad
\varPsi_\lambda = \begin{bmatrix} 1&0&0 \\ 0&\lambda&0 \\ 0&0&\lambda^{q_0+1} \end{bmatrix}.
\]
We define the groups
\begin{align*}
K&=\{\varPhi_{a,c}\mid a,c \in \bbF_q \},\\
H&=\{ \varPsi_\lambda \mid \lambda \in \bbF_q^* \}.
\end{align*}
Then, $K$ is a $2$-group of order $q^2$ and $H$ is a cyclic group of order $q-1$. Moreover, $H$ normalizes $K$, and its action fixes no nontrivial element in $K$. Their closure group is $HK$, and denoted by $A(f,q_0)$ in Higman's paper \cite{hig}. For brevity, we write $G$ in place of $A(f,q_0)$. With this change $G=HK$.  Since $H\cap H^g=1$ holds for any $g\in G\setminus H$, $G$ is a Frobenius group in its action on the set $G/H$ of right cosets. The point stabilizer is $H$ and $K$ is a regular normal subgroup. It may be noticed  that when $q=2q_0^2$ then $G$ is similar to the $1$-point stabilizer of the Suzuki group $\mathrm{Sz}(q)$ in its double transitive action on $q^2+1$ points.
A straightforward computation shows that the $H$-orbits in $K$ are
\begin{equation} \label{eq:Omegau}
\Omega_u= \{\varPhi_{a,ua^{q_0+1}} \mid a\in \bbF_q^*\}, \qquad u\in \bbF_q,
\end{equation}
and
\[\Omega_\infty = \{\varPhi_{0,c} \mid c\in \bbF_q^*\}.\]

\section{A Cayley graph arising from $G$}
\label{sec2}
For every $u\in \bbF_q$, we may build a Cayley graph in the usual way:
\[\Gamma_u=\mathrm{Cay}(K,\Omega_u \cup \Omega_{u+1}).\]
Since $\Omega_u \cup \Omega_{u+1}$ is $H$-invariant, the group $G$ induces automorphisms of $\Gamma_u$. This allows us to look at (the matrix group) $G$ as a Frobenius group  on $K=V(\Gamma_u)$. Our aim is to show that if $q$, $q_0$ and $u\in \bbF_q$ are carefully chosen then $\Aut(\Gamma_u)$ coincides with $G$. Define the set $\mathscr{U}_{q,q_0}$ of elements $u\in \bbF_q$ which satisfy both conditions:
\begin{enumerate}[(U1)]
\item $u=(1+\eta^{q_0})/(\eta+\eta^{q_0})$ for some primitive element $\eta$ of $\bbF_q$;
\item the polynomial $X^{q_0+1}+uX^{q_0}+(u+1)X+1$ has no roots in $\bbF_q$.
\end{enumerate}

Then such an appropriate choice of the triple $(q,q_0,u)$ is given in the following theorem.

\begin{theorem} \label{thm:main}
Assume that $q-1$ and $q_0^2-1$ are relatively prime. Then
\begin{enumerate}
\item[(i)] $\Gamma_u$ is connected Cayley graph.
\end{enumerate}  If, in addition,  $u\in \mathscr{U}_{q,q_0}$, then
\begin{enumerate}
\item[(ii)] $\Aut(\Gamma_u) = G$, that is, $G$ has a graphical Frobenius representation on $\Gamma_u$.
\end{enumerate}
\end{theorem}

The question whether Theorem \ref{thm:main} provides an infinite family is also answered positively.
\begin{theorem} \label{thm:nonempty}
For infinitely many $2$-powers $q$ it is true that whenever the $2$-power $q_0$ satisfies $\gcd(q-1,q_0^2-1)=1$, the set $\mathscr{U}_{q,q_0}$ is not empty.
\end{theorem}

\section{Some more properties of the abstract structure of the group $G$}


\begin{lemma}
The following hold in $K$:
\begin{enumerate}[(i)]
\item $\varPhi_{a,c}^2=\varPhi_{0,a^{q_0+1}}$ and $\varPhi_{a,c}=\varPhi_{a,c+a^{q_0+1}}$.
\item $\varPhi_{a,c}^{-1}\varPhi_{b,d}^{-1}\varPhi_{a,c}\varPhi_{b,d} = \varPhi_{0,a^{q_0}b+ab^{q_0}}$.
\item $\Omega_\infty$ consists of central involutions of $K$.
\item For each $u\in \bbF_q$, we have $\Omega_u^{-1}=\Omega_{u+1}$.
\end{enumerate}
\end{lemma}
\begin{proof}
Straightforward matrix computation.
\end{proof}

\begin{lemma} \label{lm:GHKprops}
Assume that $\gcd(q_0^2-1,q-1)=1$. Then the following hold:
\begin{enumerate}[(i)]
\item $K'=Z(K)=\{1\}\cup \Omega_\infty$.
\item $K'$ and $K/K'$ are elementary Abelian $2$-groups of order $q$.
\item For $u\in \bbF_q$, the set $\Omega_u$ generates $K$.
\item $H$ acts transitively (hence irreducibly) on $K'$ and $K/K'$.
\item The subgroup $H$ is maximal in $HK'$, which is maximal in $G$.
\end{enumerate}
\end{lemma}
\begin{proof}
By the assumption, the map $a\mapsto a+a^{q_0}$ has kernel $\bbF_2$, and, $a\mapsto a^{q_0+1}$ is a bijection of $\bbF_q^*$. Hence, any element of $\bbF_q$ can be written in the form $a^{q_0}b+ab^{q_0}$, which implies (i). For $a\in \bbF_q^*$, we have $\varPhi_{a,ua^{q_0+1}}^2=\varPhi_{0,a^{q_0+1}}$. Thus, $\Omega_\infty \subseteq \langle \Omega_u \rangle$ and (iii) follows. The rest is straightforward computation.
\end{proof}
Notice that Lemma \ref{lm:GHKprops}(iii) yields Theorem \ref{thm:main}(i),

\section{On Conditions (U1) and (U2)}
A natural key question regarding the applicability of Theorem \ref{thm:main} is the existence of some $q$ such that $\mathscr{U}_{q,q_0}$ is not empty, that is,  $\bbF_q$ contains an element $u$ satisfying both Conditions $U(1)$ and $U(2)$. Theorem \ref{thm:nonempty} states that infinitely many such $q$ exist and we are going to show how to prove it using Euler's phi function and the M\"obius function. For this purpose, we need some algebraic preparatory results stated in the next lemmas.

\begin{lemma} \label{lm:U(x)roots}
Let $q=2^f$ be a power of $2$ with odd exponent $f$. There exist at least $2(q+1)/3$ elements $u \in \bbF_q$ such that $X^{q_0+1}+uX^{q_0}+(u+1)X+1$ has no roots in $\bbF_q$.
\end{lemma}
\begin{proof}
Define the rational function
\[U(x)=\frac{x^{q_0+1}+x+1}{x^{q_0}+x}.\]
Clearly, $0$ and $1$ are never roots of $X^{q_0+1}+uX^{q_0}+(u+1)X+1$. Moreover, $X^{q_0+1}+uX^{q_0}+(u+1)X+1$ has a root in $\bbF_q$ if and only if $u=U(x)$ for some $x\in \bbF_q\setminus\{0,1\}$. Since $U(0)=U(1)=\infty$ and
\[U(x)=U\left(\frac{x+1}{x}\right) =U\left(\frac{1}{x+1}\right)\]
identically, we have $|U(\bbF_q\setminus\{0,1\})|\leq (q-2)/3$. Here we use the fact that $\bbF_4$ is not a subfield of $\bbF_q$ and $x$, $(x+1)/x$, $1/(x+1)$ are distinct elements of $\bbF_q$. 
\end{proof}

\begin{lemma} \label{lm:eulerphi}
For infinitely many odd integers $n$ holds $\varphi(2^n-1)/(2^n-1)>1/3$.
\end{lemma}
\begin{proof}
The claim follows from the asymptotic formula of \cite{Shpar}*{Theorem 3}
\[ \frac{1}{M} \sum _{1\leq m \leq M} \frac{\varphi(2^m-1)}{2^m-1}=\mu +O(M^{-1}\log M),\]
with $\mu$ is given by the absolute convergent series
\[\mu=\sum_{\text{$d$ odd}} \frac{\mu(d)}{dt_d} \approx 0.73192,\]
where $t_d$ is the multiplicative order of $2$ modulo $d$, and $\mu(d)$ is the M\"obius function; see \cite{Gathenetal}*{Theorem 4.1}.

We give a second, elementary proof based on Fermat's Little Theorem. We show that for primes $p$, $\varphi(2^p-1)/(2^p-1) \to 1$. Let $r_1,\ldots,r_k$ be the different prime factors of $2^p-1$. For $i=1,\ldots,k$, let $m_i$ be the order of $2$ modulo $r_i$. Then $m_i \mid p$ and $p=m_i$. Moreover, $2^{r_i-1}\equiv 1 \pmod{r_i}$ implies $p\mid (r_i-1)$. In fact, $p\mid (r_i-1)/2$ and $r_i=2s_ip+1$ holds for some integer $s_i\geq 1$. This implies
\[k<\log_{2p}(2^p-1)<\frac{p}{\log_2 p}.\]
Hence,
\[1>\frac{\varphi(2^p-1)}{2^p-1} = \prod_{i=1}^{k} \left(1-\frac{1}{r_i}\right)
>\left(1-\frac{1}{2p}\right)^\frac{p}{\log_2 p},\]
where the latter term converges to $1$. This proves our claim.
\end{proof}
\begin{remark}
As pointed out in \cite{MathOverflow}, much more is true: \cite{Shpar} implies that given any $\varepsilon >0$, there is a $c>0$ such that $\varphi(2^n-1)/(2^n-1)>c$ apart from a set of $n$ with upper density $<\varepsilon$.
\end{remark}

We are in a position to prove Theorem \ref{thm:nonempty}. 
By Lemma \ref{lm:eulerphi}, it suffices to show that for an arbitrary odd integer $f$ with $\varphi(2^f-1)/(2^f-1)>1/3$, $q=2^f$ fulfills the conditions of Theorem \ref{thm:nonempty}. Fix such an $f$ and choose an arbitrary integer $f_0$, coprime to $f$. Then $q_0=2^{f_0}$ satisfies $\gcd(q-1,q_0^2-1)=1$. By the choice of $f$, $\bbF_q$ has more than $(q-1)/3$ primitive elements. In our case, $x\mapsto x^{q_0-1}$ is bijective in $\bbF_q$, hence the maps
\[\eta \mapsto \eta' =\frac{1+\eta^{q_0}}{\eta+\eta^{q_0}}, \qquad
u \mapsto u'=1+\left(\frac{u}{u+1}\right)^\frac{1}{q_0-1}\]
are well-defined inverses to each other. Now, the claim follows from Lemma \ref{lm:U(x)roots}.

\section{Incidences}
\label{sec:incidences}

Recall that $\Gamma_u$ denotes the Cayley graph $\mathrm{Cay}(K,\Omega_u \cup \Omega_{u+1})$, where the vertices of $\Gamma_u$ are the elements of $K$ and $\Omega_u$ is defined in \eqref{eq:Omegau}. The identity $\varPhi_{0,0}$ of $K$ will be also denoted by $\varepsilon$. The group $G=HK$ acts on $K$, the action is induced as follows: The elements of $K$ act in the right regular action and the elements of $H$ act by conjugation. In the sequel, we identify $G$ with its permutation action on $K$, whereby some caution is required since for a subset $X$ of $K$, the point-wise stabilizer of $X$ in $G$ and the centralizer of $X$ in $G$ are in general different. As a permutation group, $G$ is a subgroup of the automorphism group $\Aut(\Gamma_u)$, and $H$ is its cyclic subgroup of order $q-1$, fixing $\varepsilon$ and preserving both $\Omega_u$ and $\Omega_{u+1}$. Formally, $\varepsilon$ is viewed as an element of $\Aut(\Gamma_u)$; nevertheless, we will also use the notation $\id$ to denote the trivial automorphism of $\Aut(\Gamma_u)$.

For any two elements $\varPhi_{a,c}, \varPhi_{b,d} \in K$ with  $\varPhi_{a,c} \varPhi_{b,d}^{-1} \in \Omega_u$, we introduce the directed edge notation $\varPhi_{a,c} \edge{u} \varPhi_{b,d}$ in $\Gamma_u$ and we refer to it as a $u$-edge. An obvious observation is that the following are equivalent:
\begin{enumerate}[(i)]
\item $\varPhi_{a,c} \edge{u} \varPhi_{b,d}$,
\item $\varPhi_{a,c} \varPhi_{b,d}^{-1} \in \Omega_u$,
\item $c+d=(a+b)^{q_0}(ua+(u+1)b)$,
\item $c+d=u(a+b)^{q_0+1}+a^{q_0}b+b^{q_0+1}.$
\end{enumerate}

Now we collect some incidences in $\Gamma_u$ which play a role in our proof.

\begin{lemma} \label{lm:incideces}
Assume $\gcd(q-1,q_0^2-1)=1$ and define
\[\eta = 1+\left(\frac{u}{u+1}\right)^\frac{1}{q_0-1}\]
for $u\in \bbF_q\setminus\{0,1\}$.
Then the following hold in $\Gamma_u$ for $a,b \neq 0$:
\begin{subequations}
\begin{alignat}{2}
\varPhi_{a,ua^{q_0+1}} &\edge{u} \varPhi_{b,ub^{q_0+1}} && \Longleftrightarrow b=\frac{a}{\eta}, \label{eq1} \\
\varPhi_{a,ua^{q_0+1}} &\edge{u+1} \varPhi_{b,ub^{q_0+1}} && \Longleftrightarrow b=a\eta, \label{eq2} \\
\varPhi_{a,(u+1)a^{q_0+1}} &\edge{u} \varPhi_{b,(u+1)b^{q_0+1}} && \Longleftrightarrow b=a \cdot \frac{\eta}{1+\eta}, \label{eq3} \\
\varPhi_{a,(u+1)a^{q_0+1}} &\edge{u+1} \varPhi_{b,(u+1)b^{q_0+1}} && \Longleftrightarrow b=a \cdot \frac{1+\eta}{\eta}, \label{eq4} \\
\varPhi_{a,ua^{q_0+1}} &\edge{u} \varPhi_{b,(u+1)b^{q_0+1}} && \Longleftrightarrow b=\frac{a}{1+\eta}, \label{eq5}  \\
\varPhi_{a,ua^{q_0+1}} & \edge{u+1} \varPhi_{b,(u+1)b^{q_0+1}} && \Longleftrightarrow \left(\frac{a}{b}\right)^{q_0+1}+u\left(\frac{a}{b}\right)^{q_0}+(u+1)\left(\frac{a}{b}\right)+1=0. \label{eq6}
\end{alignat}
\end{subequations}
\end{lemma}
\begin{proof}
\eqref{eq1}: Since $\Gamma_u$ has no loops, we may assume $a\neq b$.
\begin{align*}
\varPhi_{a,ua^{q_0+1}} \edge{u} \varPhi_{b,ub^{q_0+1}} & \Longleftrightarrow ua^{q_0+1}+ub^{q_0+1}=(a+b)^{q_0}(ua+(u+1)b)\\
& \Longleftrightarrow  0=(u+1)a^{q_0}b+uab^{q_0}+b^{q_0+1}\\
& \Longleftrightarrow  0=(u+1)\left(\frac{a}{b}\right)^{q_0}+u\left(\frac{a}{b}\right)+1\\
& \Longleftrightarrow  0=(u+1)\left(\frac{a}{b}+1\right)^{q_0}+u\left(\frac{a}{b}+1\right)\\
& \Longleftrightarrow  \left(\frac{a}{b}+1\right)^{q_0-1} = \frac{u}{u+1} = (\eta+1)^{q_0-1}\\
& \Longleftrightarrow \frac{a}{b}={\eta}.
\end{align*}
Since $(u+1)$-edges are reversed $u$-edges, we obtain \eqref{eq2} by switching $a$ and $b$ in the computation above. To show \eqref{eq4}, we replace $u$ by $u+1$ and use the computation above to obtain
\begin{align*}
\varPhi_{a,(u+1)a^{q_0+1}} \edge{u+1} \varPhi_{b,(u+1)b^{q_0+1}}
& \Longleftrightarrow  \left(\frac{a}{b}+1\right)^{q_0-1} = \frac{u+1}{u} = \left(\frac{1}{1+\eta}\right)^{q_0-1}\\
& \Longleftrightarrow \frac{a}{b}=\frac{\eta}{1+\eta}.
\end{align*}
This proves \eqref{eq3} by switching $a$ and $b$. For \eqref{eq5}:
\begin{align*}
\varPhi_{a,ua^{q_0+1}} \edge{u} \varPhi_{b,(u+1)b^{q_0+1}} & \Longleftrightarrow ua^{q_0+1}+(u+1)b^{q_0+1}=(a+b)^{q_0}(ua+(u+1)b)\\
& \Longleftrightarrow  0=(u+1)a^{q_0}b+uab^{q_0}\\
& \Longleftrightarrow  \left(\frac{a}{b}\right)^{q_0-1} = \frac{u}{u+1} = (\eta+1)^{q_0-1}\\
& \Longleftrightarrow \frac{a}{b}={1+\eta}.
\end{align*}
Finally,
\begin{align*}
\varPhi_{a,ua^{q_0+1}} \edge{u+1} \varPhi_{b,(u+1)b^{q_0+1}} & \Longleftrightarrow ua^{q_0+1}+(u+1)b^{q_0+1}=(a+b)^{q_0}((u+1)a+ub)\\
& \Longleftrightarrow 0=a^{q_0+1}+ua^{q_0}b+(u+1)ab^{q_0}+b^{q_0+1}\\
& \Longleftrightarrow 0=\left(\frac{a}{b}\right)^{q_0+1}+u\left(\frac{a}{b}\right)^{q_0}+(u+1)\left(\frac{a}{b}\right)+1,
\end{align*}
which shows \eqref{eq6}.
\end{proof}

Our next step is to describe the structure of the neighborhood of the vertex $\varepsilon$ in $\Gamma_u$. For this purpose, we recall the concept of \textit{generalized Petersen graphs} \cite{GPG}. Let $n$ and $k$ be integers with $1\leq k <n/2$, the vertex set of $GPG(n,k)$ is $\{c_1,\ldots,c_n,c_1',\ldots,c_n'\}$ and the edge set consists of all pairs of the form
\[c_ic_{i+1}, \qquad c_ic_i', \qquad c_ic_{i+k}', \qquad i\in \{1,\ldots,n\},\]
where all subscripts are to be read modulo $n$. In order to describe the automorphism group of $GPG(n,k)$, define the permutations
\begin{alignat*}{2}
\rho&:c_i\mapsto c_{i+1}, \quad &&c_i'\mapsto c_{i+1}',\\
\delta&:c_i\mapsto c_{-i}, \quad &&c_i'\mapsto c_{-i}',\\
\alpha&:c_i\mapsto c_{ki}', \quad &&c_i'\mapsto c_{ki}
\end{alignat*}
for all $i\in \{1,\ldots,n\}$. By \cite{GPG}*{Theorem 1 and 2},
\[\langle \rho, \delta \rangle \leq \Aut(GPG(n,k)) \leq \langle \rho, \delta, \alpha  \rangle\]
provided that $n\not\in \{4,5,8,10,12,24\}$. Moreover, the generators $\rho$, $\delta$, satisfy the relations $\rho^n=\delta^2=\id$, $\delta\rho\delta=\rho^{-1}$, hence, $\langle \rho, \delta \rangle$ is isomorphic to the dihedral group of order $2n$. Also, $\alpha\delta=\delta\alpha$, $\alpha^2 \in \{\id,\delta\}$, and most importantly $\alpha^{-1}\rho\alpha =\rho^k$. This implies the following lemma:
\begin{lemma} \label{lm:GPGparam}
Let $n$ be an odd integer, $n\neq 5$, and $1\leq k <n$. In $\Aut(GPG(n,k))$, the following properties hold:
\begin{enumerate}[(i)]
\item The elements of odd order form a unique cyclic normal subgroup of order $n$.
\item For $k\neq \pm 1$, no involution commutes with the cyclic normal subgroup of order $n$. \qed
\end{enumerate}
\end{lemma}

\begin{proposition} \label{pr:GPG}
Assume $\gcd(q-1,q_0^2-1)=1$ and $u\in \mathscr{U}_{q,q_0}$. Then, the neighborhood $\Omega_u \cup \Omega_{u+1}$ of $\varepsilon$ in $\Gamma_u$ is isomorphic to the generalized Petersen graph $GPG(q-1,k)$, where $u = (1+\eta^{q_0})/(\eta+\eta^{q_0})$ and the integer $k$ is defined by $1+\eta=\eta^{k+1}$.
\end{proposition}
\begin{proof}
By the choice of $u$, $\eta$ is a primitive element of $\bbF_q$. Define
\[c_i=\varPhi_{\eta^i,u\eta^{i(q_0+1)}}, \qquad c_i'=\varPhi_{\eta^i/(1+\eta),(u+1)(\eta^i/(1+\eta))^{q_0+1}}. \]
From Lemma \ref{lm:incideces}, $c_ic_{.i+1}$, $c_ic'_i$ are edges and there are no more edges in $\Omega_u$ and between $\Omega_u$ and $\Omega_{u+1}$. In $\Omega_{u+1}$, $c_i'$ and $c_j'$ are connected with an $u$-edge if and only if
\[ \frac{\eta^j}{1+\eta}=\frac{\eta^i}{1+\eta}\cdot \frac{1+\eta}{\eta} \Longleftrightarrow \eta^{j-i+1}=1+\eta=\eta^{k+1} \Longleftrightarrow j\equiv i+k \pmod{q-1}. \]
This finishes the proof.
\end{proof}

Notice that $k=\pm 1$ would imply $\eta=0$ or $1+\eta+\eta^2=0$, which is not possible if $\gcd(q-1,q_0^2-1)=1$ and $\eta$ generates $\bbF_q^*$.

\begin{corollary} \label{co:0stab1}
Assume $\gcd(q-1,q_0^2-1)=1$ and $u\in \mathscr{U}_{q,q_0}$. Let $A$ be the permutation group induced by the stabilizer $\Aut(\Gamma_u)_\varepsilon$ on $\Omega_u \cup \Omega_{u+1}$. Then $A$ is solvable, its order is either $(q-1)$, $2(q-1)$ or $4(q-1)$, and it has a unique cyclic normal subgroup of odd order $q-1$. Moreover, $\Aut(\Gamma_u)_\varepsilon$ either preserves $\Omega_u$ and $\Omega_{u+1}$, or it interchanges them.
\end{corollary}
\begin{proof}
$A$ contains the cyclic subgroup of order $q-1$ that is induced by $H$ on $\Omega_u \cup \Omega_{u+1}$. Proposition \ref{pr:GPG} and Lemma \ref{lm:GPGparam} apply.
\end{proof}

We finish this section with another property of the stabilizer of $\varepsilon$ in $\Aut(\Gamma_u)$.

\begin{lemma} \label{lm:0stab2}
Assume $\gcd(q-1,q_0^2-1)=1$ and $u\in \mathscr{U}_{q,q_0}$.
\begin{enumerate}[(i)]
\item Let $A$ be the centralizer of the commutator subgroup $K'$ in  $\Aut(\Gamma_u)$. Then $K\leq A$ and $|A:K|\leq 2$. Moreover, any element of $A\setminus K$ interchanges the sets $\Omega_u$ and $\Omega_{u+1}$.
\item Let $\alpha \in \Aut(\Gamma)$ be an involution which centralizes $H$. Then $\alpha$ fixes $\Omega_u \cup \Omega_{u+1}$ point-wise.
\end{enumerate}
\end{lemma}
\begin{proof}
(i) Obvoiusly, $K\leq A$ and $A$ is transitive. From the last sentence of Corollary \ref{co:0stab1}, an element $\alpha \in A_\varepsilon$ either preserves $\Omega_u$ and $\Omega_{u+1}$, or it interchanges them. We show that if $\alpha$ preserves $\Omega_u$ then $\alpha=\id$. This will imply $|A_\varepsilon|\leq 2$ and $|A|\leq 2q^2$. Since $\alpha$ commutes with $K'$ and fixes $\varepsilon$, it fixes all points in the orbit $\varepsilon^{K'}=\{\varepsilon\}\cup \Omega_\infty$. The elements $\varPhi_{a,ua^{q_0+1}} \in \Omega_u$ and $\varPhi_{0,d} \in K'$ satisfy both relations
\begin{alignat*}{2}
\varPhi_{a,ua^{q_0+1}} &\edge{u} \varPhi_{0,d} && \Longleftrightarrow d=0,  \\
\varPhi_{a,ua^{q_0+1}} &\edge{u+1} \varPhi_{0,d} && \Longleftrightarrow d=a^{q_0+1}.
\end{alignat*}
This means that each element in $\Omega_u$ is connected with a unique element in $\Omega_\infty$. Hence, $\alpha$ fixes all elements in $\Omega_u$. As each $K'$-orbit contains a unique element in $\Omega_u$, we see that each $K'$-orbit is preserved. Once again, $\alpha$ commutes with $K'$ and fixes an element in each $K'$-orbit. Therefore, $\alpha$ fixes all points in each $K'$-orbit.

(ii) As $\varepsilon$ is the unique fixed point of $H$, $\varepsilon^\alpha=\varepsilon$ and $\alpha$ leaves the neighborhood $\Omega_u \cup \Omega_{u+1}$ of $\varepsilon$ invariant. By Lemma \ref{lm:GPGparam}(ii), the restriction of $\alpha$ to $\Omega_u \cup \Omega_{u+1}$ cannot have order $2$, therefore, it must be trivial.
\end{proof}

\section{Imprimitivity}

In this section we show that an appropriate choice of $u\in \bbF_q$ ensures that $\Aut(\Gamma_u)$ cannot act primitively on the set of vertices of $\Gamma_u$. We recall that a primitive permutation group $G$ is of \emph{affine type} if it has an abelian regular normal subgroup, which is necessarily elementary abelian of order $r^n$ for some prime $r$. In this case $G$ is embedded in the affine group $AGL(n,r)$ with the socle being the translation subgroup. Its stabiliser of $0\in \bbF_r^n$ is a subgroup of $GL(n,r)$ which acts irreducibly on $\bbF_r^n$. For our purpuse, a useful tool is the following result by Guralnick and Saxl.

\begin{proposition}[Guralnick and Saxl \cite{GuSa}] \label{prop:GS}
Let $G$ be a primitive permutation group of degree $2^n$. Then either $G$ is of affine
type, or $G$ has a unique minimal normal subgroup $N=S\times \cdots \times S = S^t$, $t\ge
1$, $S$ is a non-abelian simple group, and one of the following holds:
\begin{enumerate}
\item[(i)] $S=A_m$, $m=2^e\geq 8$, $n=te$, and the $1$-point stabilizer in $N$ is $N_1 =
A_{m-1}\times \cdots \times A_{m-1}$, or
\item[(ii)] $S=PSL(2,p)$, $p = 2^e -1\geq 7$ is a Mersenne prime, $n=te$,
and the $1$-point stabilizer in $N$ is the direct product of maximal parabolic subgroups
each stabilizing a $1$-space.
\end{enumerate}
\end{proposition}

\begin{lemma}\label{lm:capT}
Let $G$ be a group acting transitively on the set $X$. For $x\in X$ and let $H=G_x$ be the stabilizer of $x$ in $G$.
\begin{enumerate}[(i)]
\item For $y\in X$, choose $g\in G$ such that $y=x^g$. Then the subgroup of $H$, fixing the $H$-orbit of $y$ point-wise, coincides with $\cap_{h\in H} H^{gh}$.
\item If $G$ is $2$-transitive on $X$ then $\cap_{h\in H} H^{gh}$ is either $H$ or $\{1\}$, depending upon whether $g\in H$ or $g\not\in H$.
\end{enumerate}
\end{lemma}
\begin{proof}
If $y'\in y^H$, then $y'=y^h=x^{gh}$ for some $h\in H$. Hence, for the stabilizer we have $G_{y'}=G_x^{gh}=H^{gh}$. Therefore, the point-wise stabilizer of $y^H$ is $\cap_{y'\in y^H} G_{y'}=\cap_{h\in H} H^{gh}$. This proves (i). Clearly, if $g\in H$ then $\cap_{h\in H} H^{gh}=H$. If $g\in G\setminus H$ then $x\neq y=x^g$ and $\cap_{h\in H} H^{gh}$ fixes all points in $\{x\}\cup y^H$. The latter set is $X$ if $G$ is $2$-transitive.
\end{proof}

\begin{lemma}\label{lm:onlyaffine}
Assume $\gcd(q-1,q_0^2-1)=1$ and $u\in \mathscr{U}_{q,q_0}$. If $\Aut(\Gamma_u)$ acts primitively on $\Gamma_u$, then its action is of affine type.
\end{lemma}
\begin{proof}
Let us assume on the contrary that $\Aut(\Gamma_u)$ is not of affine type. Let $N$ be its unique minimal normal subgroup. With the notation in Proposition \ref{prop:GS}, we have $N=S^t$ where either $S=A_m$, $m\geq 8$, or $S=PSL(2,p)$, with a Mersenne prime $p=m-1\geq 7$. In both cases, $S$ has a $2$-transitive action on $m$ points. Moreover, if $B$ is the $1$-point stabilizer in $S$, then the point stabilizer of $\varepsilon=\varPhi_{0,0}$ in $N$ is $N_\varepsilon=B^t$. For $(g_1,\ldots,g_t)\in S^t$ take a generic vertex $y=\varepsilon^{(g_1,\ldots,g_t)}$ of $\Gamma_u$. Let $Y$ be the $B^t$-orbit of $y$. By Lemma \ref{lm:capT}(i) the point-wise stabilizer of $Y$ is
\[(\cap_{b\in B} B^{g_1b}) \times \cdots \times (\cap_{b\in B} B^{g_tb}).\]
By Lemma \ref{lm:capT}(ii), each factor is either $\{1\}$ or $B$, depending upon whether $g_i \in B$ or not. Thus, the point-wise stabilizer of $Y$ in $B^t$ is $B^{t_0}$, where $0 \leq t_0\leq t$, and $t_0=t$ occurs if and only if $Y=\{\varepsilon\}$. Therefore, the $B^t$ induces a permutation group on $Y$ which is isomorphic to $B^{t_1}$, where $t_1=t-t_0$. Furthermore, $t_1=0$ if and only if $Y=\{\varepsilon\}$.

The stabilizer $N_\varepsilon$ acts on $\Omega_u \cup \Omega_{u+1}$. Let $Y$ be a nontrivial $N_\varepsilon$-orbit contained in $\Omega_u \cup \Omega_{u+1}$. If $S=A_m$, then $N_\varepsilon$ induces a nonsolvable group of automorphisms of $\Omega_u \cup \Omega_{u+1}$. If $S=PSL(2,p)$, then $|B|=p(p-1)/2$, and $N_\varepsilon$ induces a noncyclic group of odd order on $\Omega_u \cup \Omega_{u+1}$. Both possibilities are inconsistent with Corollary \ref{co:0stab1}.
\end{proof}

We are now able to prove the imprimitivity of $\Aut(\Gamma_u)$. 

\begin{proposition} \label{pr:imprimitive}
Assume $\gcd(q-1,q_0^2-1)=1$ and $u\in \mathscr{U}_{q,q_0}$. Then, $\Aut(\Gamma_u)$ acts imprimitively on $\Gamma_u$.
\end{proposition}
\begin{proof}
As before, $G$ is identified with its permutation action on $\Gamma_u$. In particular, we consider $H$, $K$ as subgroups of $\Aut(\Gamma_u)$. At the same time, $K$ is the set of vertices of $\Gamma_u$.

Assume on the contrary that $\Aut(\Gamma_u)$ is primitive, hence of affine type by Lemma \ref{lm:onlyaffine}. Let $N$ be the unique minimal normal subgroup of $\Aut(\Gamma_u)$. Then $N$ is a regular elementary abelian $2$-group. Since $H$ has odd order, $N$ decomposes into the direct product of $H$-invariant subgroups. For any $1\neq h \in H$ and $1\neq n\in N$, $h$ has a unique fixed point, while $n$ has no fixed point. Hence $nh\neq hn$. Therefore $N=A_1\times A_2$ where $A_i$ is an elementary abelian group of order $q$ and $H$ acts regularly on $A_i\setminus \{1 \}$, $i=1,2$. Consider the subgroup $M=N_{NK}(K)$. Since $NK$ is nilpotent, we have $K\lneq M$ and $K'\triangleleft M$. The latter implies $K'\cap Z(M) \neq \{1\}$. Since both $K'$ and $Z(M)$ are $H$-invariant while $H$ acts regularly on $K'\setminus\{1\}$, we have $K'\leq Z(M)$. By Lemma \ref{lm:0stab2}, $|M:K|=2$. On the one hand, $M=(M\cap N)K$. On the other hand, $N\cap K$ is an $H$-submodule of $M\cap N$. By Maschke's Theorem \cite{CurtisReiner}*{(10.8)} applied to $M\cap N$, viewed as a $\bbF_2$-vector space, there is an $H$-invariant subgroup $B$ in $M\cap N$ such that $M\cap N = B \times (N\cap K)$. Therefore,
\[B \cong (M\cap N) / (N \cap K) \cong (M\cap N)K/K \cong = M/K \cong \bbF_2,\]
that is, the nontrivial element of $B\leq N$ commutes with $H$, a contradiction.
\end{proof}

\section{Proof of the main result Theorem \ref{thm:main}(ii)}

In this section, we complete the proof of Theorem \ref{thm:main}. As before, $G$ is identified with its permutation action on $\Gamma_u$. From Proposition \ref{pr:imprimitive}, we know that $A=\Aut(\Gamma_u)$ acts imprimitively on $\Gamma_u$. We claim that the only nontrivial blocks of imprimitivity of $A$ are the cosets of the commutator subgroup $K'$ of $K$. Or equivalently, $K'$ is the only nontrivial block containing $\varepsilon$. Let $B$ be an arbitrary nontrivial block of imprimitivity of $A$ which contains $\varepsilon$. Then the stabilizer of the set $B$ in $G$ is a subgroup $G_B$ of $G$, lying properly between $H$ and $G$. By Lemma \ref{lm:GHKprops}(v), $G_B=HK'$ and $B=K'$, which proves the claim. The next two lemmas describe the point-wise stabilizer of $K'$ in $A$.

\begin{lemma} \label{lm:ptstabK'}
Let $E$ be the point-wise stabilizer of $K'=\{\varepsilon\}\cup \Omega_\infty$ in $\Aut(\Gamma_u)$. Then $E$ is either trivial or it is an elementary abelian $2$-group which fixes all pairs $\{\varPhi_{a,c},\varPhi_{a,c}^{-1}\}$.
\end{lemma}
\begin{proof}
By the observations made prior to Lemma \ref{lm:incideces} show
\[\varPhi_{a,c} \edge{u} \varPhi_{0,d} \Longleftrightarrow d=c+ua^{q_0+1}\]
for all $a,c,d \in \bbF_q$. Thus, any vertex $\varPhi_{a,c}$, $a\neq 0$, is $u$-connected to a unique element $\varPhi_{0,d_1}$ of $K'$ and $(u+1)$-connected to a unique element $\varPhi_{0,d_2}$ of $K'$, where $d_1=c+ua^{q_0+1}$ and $d_2=c+(u+1)a^{q_0+1}$. If $d_1$ and $d_2$ are distinct nonzero elements, then $\varPhi_{0,d_1}$, $\varPhi_{0,d_2}$ are distinct vertices in $\Omega_\infty$, whose common neighbors are $\varPhi_{a,c}$ and $\varPhi_{a,c}^{-1}=\varPhi_{a,c+a^{q_0+1}}$, where
\[a=(d_1+d_2)^{\frac{1}{q_0+1}} \quad \text{and}\quad c\in \{d_1+ua^{q_0+1}, d_2+ua^{q_0+1}\}.\]
This shows that any automorphism of $\Gamma_u$, which fixes $\Omega_\infty$ point-wise, must leave the pair $\{\varPhi_{a,c},\varPhi_{a,c}^{-1}\}$ invariant. It follows that $E$ either trivial or has exponent $2$ and in the latter case $E$ is elementary abelian.
\end{proof}

Actually, $E$ is trivial by the following lemma.

\begin{lemma} \label{lm:Eistrivial}
The only automorphism that fixes $\{\varepsilon\} \cup \Omega_\infty$ point-wise is the identity.
\end{lemma}
\begin{proof}
Let $E$ be defined as in Lemma \ref{lm:ptstabK'}. Since $HK'$ preserves the set of vertices in $K'$, $HK'$ normalizes $E$. Assume on the contrary that $E\neq \{1\}$, then $C_E(K') \neq \{1\}$ is $H$-invariant. Since $K'$ acts regularly on itself, $E\cap K'=\{1\}$. We apply Lemma \ref{lm:0stab2}(i) to conclude that $|C_E(K')|=2$. This means that there is a unique involutory automorphism $\alpha \in A$ which centralizes both $K'$ and $H$. Now, Lemma \ref{lm:0stab2}(ii) implies that $\alpha$ fixes $\Omega_u\cup \Omega_{u+1}$ point-wise. Finally, Lemma \ref{lm:0stab2}(i) yields $\alpha \in K$, a contradiction. 
\end{proof}

Let us now focus on the point stabilizer $A_\varepsilon$ of $\varepsilon$ in $A=\Aut(\Gamma_u)$. Clearly, $A_\varepsilon$ leaves $\Omega_u \cup \Omega_{u+1}$ invariant. Moreover, by the imprimitivity of $A$, $A_\varepsilon$ preserves $\Omega_\infty$ as well. Since any element of $\Omega_u$ is connected with a unique element of $\Omega_\infty$, each automorphism fixing all points in $\{\varepsilon\} \cup \Omega_u \cup \Omega_{u+1}$ fixes all points in $\Omega_\infty$. Hence by Lemma \ref{lm:Eistrivial}, the action of $A_\varepsilon$ on $\Omega_u \cup \Omega_{u+1}$ is faithful and the possibilities for $|A_\varepsilon|$ are $q-1$, $2(q-1)$ or $4(q-1)$ by Corollary \ref{co:0stab1}.

Let $S$ denote the stabilizer of the set $K'$ in $A$. On the one hand, $HK'\leq S$, hence $S$ is transitive on $K'$. On the other hand, $A_\varepsilon \leq S$ since $K'$ is a block of imprimitivity. Therefore, $A_\varepsilon = S_\varepsilon$, and
\[|S|=q|A_\varepsilon|\in \{q(q-1),2q(q-1),4q(q-1)\}.\]
This implies that $S$ induces a $2$-transitive solvable permutation group $\bar{S}$ on $K'$. Since the order of $K'$ is a power of $2$, Huppert's Theorem \cite{HuppertBlackburn}*{Theorem XII.7.3} yields that $\bar{S}$ is similar to a subgroup of the group $A\Gamma{}L(1,q)$ of all semilinear mappings
\[z\mapsto az^\alpha+b,\qquad a,b\in \bbF_q, a\neq 0, \alpha \in\Aut(\bbF_q)\]
on $\bbF_q$. Here, $|A\Gamma{}L(1,q)|=fq(q-1)$ for $q=2^f$. Since $\gcd(q-1,q_0^2-1)=1$, $f$ is odd, and the only possibility for the cardinality of $\bar{S}$ is $q(q-1)$. We apply Lemma \ref{lm:Eistrivial} once more to conclude that $|S|=q(q-1)$, which implies $A_\varepsilon = H$ and $A=HK=G$. This finishes the proof of Theorem \ref{thm:main}(ii).



\end{document}